\documentclass[a4paper, 10pt, notitlepage]{article}

\usepackage{amsthm, amsmath, amssymb, latexsym}
\usepackage{mathrsfs}

\theoremstyle{plain}
\newtheorem{theorem}{Theorem}

\newtheorem{corollary}{Corollary}
\newtheorem{proposition}{Proposition}

\theoremstyle{definition}
\newtheorem{definition}{Definition}

\theoremstyle{remark}
\newtheorem{remark}{Remark}

\DeclareMathOperator{\dom}{dom}
\DeclareMathOperator{\argmin}{arg\,min}

\author{M.V. Dolgopolik}
\title{Smooth Exact Penalty Functions: A General Approach}

\begin{document}

\maketitle

\begin{abstract}
In this article we present a new perspective on the smooth exact penalty function proposed by Huyer and Neumaier
that is becoming more and more popular tool for solving constrained optimization problems. Our approach to Huyer and
Neumaier's exact penalty function allows one to apply previously unused tools (namely, parametric
optimization) to the study of the exactness of this function. We give a new simple proof of the local exactness of
Huyer and Neumair's penalty function that significantly generalizes all similar results existing in the literature. We
also obtain new necessary and sufficient conditions for the global exactness of this penalty function.
\end{abstract}

\section{Introduction}

The method of exact penalty functions \cite{DiPilloGrippo, BurkeExPen, Demyanov} is a powerful tool for solving various
constrained optimization problems. However, as it is well-known, exact penalty functions are usually nonsmooth, even in
the case when the original problem is smooth. This obstacle makes it impossible to apply (without some transformation
of the problem) well-developed and extensively studied methods of smooth unconstrained optimization to minimization of
an exact penalty function.

Huyer and Neumaier in \cite{HuyerNeumaier} proposed a new approach to exact penalization that allows one to overcome
nonsmoothness of exact penalty functions. Namely, let the original constrained optimization problem have the form
\begin{equation} \label{NonlinearProgr}
  \min f(x) \quad \text{subject to} \quad F(x) = 0, \quad x \in [\underline{x}, \overline{x}],
\end{equation}
where $f \colon \mathbb{R}^n \to \mathbb{R}$ and $F \colon \mathbb{R}^n \to \mathbb{R}^m$ are smooth functions,
$\underline{x}, \overline{x} \in \mathbb{R}^n$ are given vectors, and
$$
  [\underline{x}, \overline{x}] = \big\{ x = (x_1, \ldots, x_n) \in \mathbb{R}^n \mid 
  \underline{x}_i \le x_i \le \overline{x}_i \quad \forall i \in \{ 1, \ldots, n \} \big\}.
$$
The new approach consists in the introduction of an additional variable $\varepsilon \ge 0$ in the following way. 
Choose $w \in \mathbb{R}^m$, and note that problem (\ref{NonlinearProgr}) is equivalent to the problem
\begin{equation} \label{AugmProblem}
  \min_{x, \varepsilon} f(x) \quad \text{subject to} \quad F(x) = \varepsilon w, \quad \varepsilon = 0, \quad 
  x \in [\underline{x}, \overline{x}].
\end{equation}
Then one defines the new ``smooth'' penalty function for the augmented problem (\ref{AugmProblem}) as follows
\begin{equation} \label{SmoothExPenFuncNLP}
  F_{\lambda}(x, \varepsilon) = \begin{cases}
    f(x), & \text{ if } \varepsilon = \Delta (x, \varepsilon) = 0, \\
    f(x) + \frac{1}{2 \varepsilon} \frac{\Delta(x, \varepsilon)}{1 - q \Delta(x, \varepsilon)} + \lambda
    \beta(\varepsilon), & \text{ if } \varepsilon > 0, \: \Delta(x, \varepsilon) < q^{-1}, \\
    + \infty, &  \text{otherwise}.
  \end{cases}
\end{equation}
where $\lambda \ge 0$ is the penalty parameter, $\Delta(x, \varepsilon) = \| F(x) - \varepsilon w \|^2$ is the
constraint violation measure, $\beta \colon [0, \overline{\varepsilon}] \to [0, +\infty)$ with $\beta(0) = 0$ is the
penalty term, $q > 0$ and $\overline{\varepsilon} > 0$ are some prespecified thresholds. Finally, one replaces the
augmented problem (\ref{AugmProblem}) with the penalized problem 
\begin{equation} \label{FirstPenProblem}
  \min_{x, \varepsilon} F_{\lambda}(x, \varepsilon) \quad \text{subject to} \quad 
  (x, \varepsilon) \in [\underline{x}, \overline{x}] \times [0, \overline{\varepsilon}].
\end{equation}
Observe that the penalty function $F_{\lambda}(x, \varepsilon)$ is smooth for any $\varepsilon \in (0,
\overline{\varepsilon)}$ and $x$ such that $0 < \Delta(x, \varepsilon) < q^{-1}$ provided the function $\beta$ is smooth
on $(0, \overline{\varepsilon})$. Furthermore, it was proven in \cite{HuyerNeumaier} that under a standard assumption
(namely, constraint qualification) the penalty function $F_{\lambda}(x, \varepsilon)$ is locally exact. In other
words, $(x^*, \varepsilon^*)$ is a point of local minimum of problem (\ref{FirstPenProblem}) if and only if 
$\varepsilon^* = 0$ and $x^*$ is a point of local minimum of problem (\ref{NonlinearProgr}). Consequently,
one can apply methods of smooth unconstrained minimization to penalized problem (\ref{FirstPenProblem}) in order to
find a solution of initial constrained optimization problem (\ref{NonlinearProgr}).

Later on, Huyer and Neumaier's approach was generalized \cite{WangMaZhou, Bingzhuang} and successfully applied to
various constrained optimization problems \cite{MaLiYiu, LinWuYu}, including some optimal control problems
\cite{LiYu, JianLin, LinLoxton}. However, it should be noted that the existing proofs of the exactness of the smooth
penalty function (\ref{SmoothExPenFuncNLP}) and its various generalizations are quite complicated, and overburdened by
technical details that overshadow the understanding of the technique of smooth exact penalty functions. Also, 
the question of when problem (\ref{NonlinearProgr}) is actually equivalent to problem (\ref{FirstPenProblem})
in terms of globally optimal solutions (in this case the penalty function $F_{\lambda}(x, \varepsilon)$ is called exact)
has not been discussed in the literature.

The aim of this article is to present a new perspective on the method of smooth exact penalty functions proposed by
Huyer and Neumaier. This perspective allows one to apply previously unused tools (namely, parametric optimization) to
the study and construction of smooth exact penalty functions. It also helped us to essentially simplify the proof of
exactness of these functions. Another aim of this articles is to provide first necessary and sufficient conditions for
Huyer and Neumaier's penalty function to be (globally) exact.

The paper is organised as follows. In Section~\ref{SctnMotivation} we describe a new approach to smooth exact penalty
functions. Some general results that draw a connection between the exactness of the new penalty function and some
properties of a perturbed optimization problem are presented in Section~\ref{SctnPerturbation}. 
In Section~\ref{Sctn_ProofOfExactness} we give a new simple proof of the local exactness of the new penalty function
that significantly generalizes all results on the local exactness of Huyer and Neumaier's penalty functions existing in
the literature. We also provide new simple sufficient conditions for this penalty function to be globally exact.

\section{How to Construct a Smooth Exact Penalty Function?}
\label{SctnMotivation}

Two main approaches are usually used for the study of exact penalty functions: the direct approach that is based on the
use of error bounds and metric regularity, and the indirect one that relies on the analysis of a perturbed optimization
problem. In the indirect approach \cite{Clarke, ClarkeBook, Burke, Uderzo}, a perturbation of the initial problem is
introduced as a tool for the study of a penalty function that has already been defined. However, one can introduce
perturbation in order to construct a penalty function.

Namely, define the perturbed objective function for problem (\ref{NonlinearProgr}) as follows:
$$
  g(x, \mu) = \begin{cases}
    f(x), & \text{ if } \mu = \Delta (x, \mu) = 0, \\
    f(x) + \frac{1}{\mu} \frac{\Delta(x, \mu)}{1 - q \Delta(x, \mu)}, &
    \text{ if } \mu > 0, \: \Delta(x, \mu) < q^{-1}, \\
    + \infty, & \text{otherwise},
  \end{cases}
$$
where $\mu \ge 0$ is a perturbation parameter. Consider the perturbed optimization problem
\begin{equation} \label{PerturbedProblem}
  \min_x g(x, \mu) \quad \text{subject to} \quad x \in [\underline{x}, \overline{x}].
\end{equation}
It is clear that the problem above with $\mu = 0$ is equivalent to problem (\ref{NonlinearProgr}). Moreover, the
penalization of the constraint $F(x) = 0$ is achieved via the introduced perturbation. As a second step, note that 
the perturbed problem is equivalent to the problem
$$
  \min_{x, \mu} g(x, \mu) \quad \text{subject to} \quad x \in [\underline{x}, \overline{x}], \quad \mu = 0.
$$
Finally, inroduce a penalty function for the problem above that penalizes only the constraint on the perturbation
parameter, i.e. $\mu = 0$. This penalty function has the form
\begin{equation} \label{PenaltyFunctionForPertProb}
  F_{\lambda}(x, \mu) = g(x, \mu) + \lambda \beta(\mu),
\end{equation}
and it is smooth for any $\mu \in (0, \overline{\varepsilon})$ and $x \in \mathbb{R}^n$ such that 
$0 < \Delta(x, \mu) < q^{-1}$. Thus, the fact that the nonlinear constraints are taken into account via perturbation
(not penalization), and the fact that the penalty term is constructed only for a simple one dimensional constraint on
the perturbation parameter helped us to avoid nonsmoothness that usually arises due to the restrictive requirements on
the penalty term of an exact penalty function.

In the following section, we develop the approach discussed above in the general case, and demonstrate that the
exactness of penalty function (\ref{PenaltyFunctionForPertProb}) is directly connected with some properties of perturbed
problem (\ref{PerturbedProblem}).

\section{Exact Penalty Function for a Perturbed Optimization Problem}
\label{SctnPerturbation}

Let $X$ be a topological space, $f \colon X \to \mathbb{R} \cup \{ + \infty \}$ be a given function, and 
$M$, $A \subset X$ be nonempty sets such that $M \cap A \ne \emptyset$. Hereafter, we study the following optimization
problem:
$$
  \min f(x) \quad \text{subject to} \quad x \in M, \quad x \in A.	\eqno{(\mathcal{P})}
$$
Denote by $\Omega = M \cap A$ the set of feasible points of this problem. Denote also  $\mathbb{R}_+ = [0, +\infty)$ and
$\dom f = \{ x \in X \mid f(x) < +\infty \}$. We suppose that $f$ is bounded below on $\Omega$.

Introduce a metric space of perturbation parameters $(P, d)$, and a perturbed objective function 
$g \colon X \times P \to \mathbb{R} \cup \{ + \infty \}$ such that there exists $\mu_0 \in P$ for which the following
conditions are satisfied:
\begin{enumerate}
\item{$g(x, \mu_0) = f(x)$ for any $x \in \Omega$ (\textbf{consistency condition});}

\item{$\argmin_{x \in A} g(x, \mu_0) = \argmin_{x \in \Omega} f(x)$ (\textbf{exact penalization condition}).}
\end{enumerate}
With the use of the exact penalization condition one gets that the problem ($\mathcal{P}$) is equivalent (in terms of
globally optimal solutions) to the following optimization problem:
\begin{equation} \label{EquivProblem}
  \min_{x, \mu} g(x, \mu) \quad \text{subject to} \quad x \in A, \quad \mu = \mu_0.
\end{equation}
Furthermore, the consistency condition guarantees that if $(x_0, \mu_0)$ with $x_0 \in \Omega$ is a point of local
minimum of the above problem, then $x_0$ is a point of local minimum of the problem ($\mathcal{P}$).

We apply the exact exact penalization technique to treat the problem (\ref{EquivProblem}). Namely, choose a function
$\beta \colon \mathbb{R}_+ \to \mathbb{R}_+ \cup \{ +\infty \}$ such that $\beta(t) = 0$ iff $t = 0$. For any 
$\lambda \ge 0$ define \textit{the penalty function}
$$
  F_{\lambda}(x, \mu) = g(x, \mu) + \lambda \beta(d(\mu, \mu_0))
$$
and consider the following penalized problem
\begin{equation} \label{PenalProblem}
  \min_{x, \mu} F_{\lambda}(x, \mu) \quad \text{subject to} \quad x \in A.
\end{equation}
Observe that the function $\lambda \to F_{\lambda}(x, \mu)$ is non-decreasing. 

Our aim is to study a relation between local/global minimizers of the initial problem ($\mathcal{P}$) and
local/global minimizers of the problem (\ref{PenalProblem}) in the context of the theory of exact penalty functions.
To this end, recall the definition of exact penalty function.

\begin{definition}
Let $x^* \in \dom f$ be a point of local minimum of the problem ($\mathcal{P}$). The penalty function $F_{\lambda}$ is
called (locally) \textit{exact} at the point $x^*$ (or, to be more precise, at the point $(x^*, \mu_0)$) if there
exists $\lambda \ge 0$ such that $(x^*, \mu_0)$ is a point of local minimum of the problem (\ref{PenalProblem}). The
greatest lower bound of all such $\lambda$ is denoted by $\lambda(x^*)$.
\end{definition}

\begin{definition}
The penalty function $F_{\lambda}$ is said to be (globally) \textit{exact}, if there exists $\lambda \ge 0$ such that
$F_{\lambda}$ attains a global minimum on the set $A \times P$, and if $(x^*, \mu^*) \in A \times P$ is a globally
optimal solution of the problem (\ref{PenalProblem}), then $\mu^* = \mu_0$. The greatest lower bound of all such
$\lambda \ge 0$ is denoted by $\lambda^*(g, \beta)$.
\end{definition}

\begin{remark}
{(i) Note that if $(x^*, \mu^*) \in A \times P$ is a globally optimal solution of the problem (\ref{PenalProblem}) with 
$\mu^* = \mu_0$, then $x^*$ is a globally optimal solution of the problem ($\mathcal{P}$) by virtue of the exact
penalization condition on the function $g(x, \mu)$, and the fact that $F_{\lambda}(x, \mu_0) = g(x, \mu_0)$. Thus, 
the penalty function $F_{\lambda}$ is globally exact iff there exists $\lambda \ge 0$ such that the problem
(\ref{PenalProblem}) is equivalent to the problem ($\mathcal{P}$) in terms of globally optimal solutions.
}

\noindent{(ii) It is easy to verify that if $F_{\lambda}$ is globally exact, then for any 
$\lambda > \lambda^*(g, \beta)$ the function $F_{\lambda}$ attains a global minimum on the set $A \times P$, and if
$(x^*, \mu^*) \in A \times P$ is a globally optimal solution of the problem (\ref{PenalProblem}), then $\mu^* = \mu_0$.
}
\end{remark}

Our aim is to show that the exactness of the penalty function $F_{\lambda}$ is closely related to some properties of the
perturbed optimization problem
$$
  \min_{x} g(x, \mu) \quad \text{subject to} \quad x \in A   \eqno{(\mathcal{P}_{\mu})}.
$$
Denote by $h(\mu) = \inf_{x \in A} g(x, \mu)$ \textit{the optimal value function} of this problem.

Recall that the problem ($\mathcal{P}_{\overline{\mu}}$), $\overline{\mu} \in P$, is said to be $\beta$-\textit{calm}
at a point $x^* \in A$ if there exist $\lambda \ge 0$, $r > 0$, and a neighbourhood $U$ of $x^*$ such that
$$
  g(x, \mu) - g(x^*, \overline{\mu}) \ge - \lambda \beta(d(\mu, \overline{\mu})) \quad 
  \forall x \in U \cap A \quad \forall \mu \in B(\overline{\mu}, r),
$$
where $B(\overline{\mu}, r) = \{ \mu \in P \mid d(\mu, \overline{\mu}) \le r \}$.

\begin{remark}
If $\beta(t) \equiv t$, then the concept of $\beta$-calmness coincides with the well-known concept of
calmness of a perturbed optimization problem \cite{Clarke, ClarkeBook, Burke, Uderzo}.
\end{remark}

The following propositions describes a connection between the exactness of the penalty function $F_{\lambda}$ and
the calmness of the perturbed problem ($\mathcal{P}_{\mu}$) (cf. an analogous result for classical exact penalty
functions in~\cite{Burke}).

\begin{proposition} \label{PrpLocalExNSC}
Let $x^* \in \dom f$ be a point of local minimum of the problem ($\mathcal{P}$). Then the penalty function $F_{\lambda}$
is exact at $x^*$ if and only if the problem ($\mathcal{P}_{\mu_0}$) is $\beta$-calm at $x^*$.
\end{proposition}

\begin{proof}
The validity of the proposition follows from the fact that the inequality
$$
  g(x, \mu) - g(x^*, \mu_0) \ge - \lambda \beta(d(\mu, \mu_0)) \quad 
  \forall x \in U \cap A \quad \forall \mu \in B(\overline{\mu}, r),
$$
holds true for some $\lambda \ge 0$, $r \ge 0$, and a neighbourhood $U$ of $x^*$ iff for any $x \in U \cap A$ and 
$\mu \in B(\mu_0, r)$ one has
$$
  F_{\lambda}(x, \mu) = g(x, \mu) + \beta(d(\mu, \mu_0)) \ge g(x, \mu_0) = F_{\lambda}(x, \mu_0),
$$
i.e. iff $(x^*, \mu_0)$ is a point of local minimum of $F_{\lambda}$ on the set $A$.  
\end{proof}

Let us turn to the study of global exactness. We need an auxiliary definition. The optimal value function $h$ is called
$\beta$-\textit{calm} from below at a point $\overline{\mu} \in P$, if 
$$
  \liminf_{\mu \to \mu_0} \frac{h(\mu) - h(\overline{\mu})}{\beta(d(\mu, \overline{\mu}))} > - \infty
$$
(cf.~the definition of calmness from below in \cite{Penot}).

\begin{theorem} \label{ThmGlobExNSC}
Let $\beta$ be strictly increasing. For the penalty function $F_{\lambda}$ to be globally exact it is necessary and
sufficient
that the following assumptions hold true:
\begin{enumerate}
\item{there exists a globally optimal solution of the problem ($\mathcal{P}$);}

\item{the optimal value function $h(\mu)$ is $\beta$-calm from below at $\mu_0$;} 

\item{there exists $\lambda_0 \ge 0$ such that $F_{\lambda_0}$ is bounded below on the set $A \times P$.}
\end{enumerate}
\end{theorem}

\begin{proof}
Necessity. Since the penalty function $F_{\lambda}$ is globally exact, then it attains a global minimum on the set $A
\times P$ for some $\lambda \ge 0$. Therefore, in particular, $F_{\lambda}$ is bounded below on $A \times P$.

Fix a sufficiently large $\lambda \ge 0$, and a global minimizer $(x^*, \mu^*)$ of $F_{\lambda}$ on the set 
$A \times P$. Due to the exactness of $F_{\lambda}$ one has $\mu^* = \mu_0$. Therefore, as it was mentioned above, $x^*$
is a globally optimal solution of the problem ($\mathcal{P}$). Thus, the first assumption is valid as well.

Observe that
$$
  g(x^*, \mu_0) \ge h(\mu_0) := \inf_{x \in A} g(x, \mu_0) \ge \inf_{(x, \mu) \in A \times P} F_{\lambda}(x, \mu) = 
  g(x^*, \mu_0).
$$
Hence $h(\mu_0) = g(x^*, \mu_0)$. The exactness of $F_{\lambda}$ implies that for all $x \in A$ and $\mu \in P$ the
one has
$$
  g(x, \mu) + \lambda \beta(d(\mu, \mu_0)) = F_{\lambda}(x, \mu) \ge F_{\lambda}(x^*, \mu_0) = g(x^*, \mu_0) = h(\mu_0)
$$
or, equivalently,
$$
  g(x, \mu) - h(\mu_0) \ge - \lambda \beta(d(\mu, \mu_0)) \quad \forall x \in A \quad \forall \mu \in P.
$$
Since the inequality above holds true for all $x \in A$, one obtains that
$$
  h(\mu) - h(\mu_0) \ge - \lambda \beta(d(\mu, \mu_0)) \quad \forall \mu \in P,
$$
which implies the $\beta$-calmness from below of $h$ at $\mu_0$.

Sufficiency. Let $x^*$ be a globally optimal solution of the problem ($\mathcal{P}$). Then taking into account the exact
penalization condition on the function $g(x, \mu)$ one gets that $h(\mu_0) = g(x^*, \mu_0)$, i.e. $x^*$ is a point of
global minimum of the function $x \to g(x, \mu_0)$.

From the fact that the optimal value function $h$ is $\beta$-calm from below at $\mu_0$ it follows that there exist
$\lambda_1 \ge 0$ and $\delta > 0$ such that
$$
  h(\mu) - h(\mu_0) \ge - \lambda_1 \beta(d(\mu, \mu_0)) \quad \forall \mu \in B(\mu_0, \delta).
$$
Hence for any $x \in A$ one has
$$
  g(x, \mu) - g(x^*, \mu_0) \ge h(\mu) - h(\mu_0) \ge - \lambda_1 \beta(d(\mu, \mu_0)) \quad 
  \forall \mu \in B(\mu_0, \delta)
$$
or, equivalently, for all $(x, \mu) \in A \times B(\mu_0, \delta)$ one has
$$
  F_{\lambda_1}(x, \mu) = g(x, \mu) + \lambda_1 \beta(d(\mu, \mu_0)) \ge g(x^*, \mu_0) = F_{\lambda_1}(x^*, \mu_0).
$$
On the other hand, if $x \in A$ and $\mu \notin B(\mu_0, \delta)$, then for any $\lambda \ge \lambda_2$, where
$$
  \lambda_2 = \lambda_0 + \frac{g(x^*, \mu_0) - c}{\beta(\delta)}, \quad
  c = \inf_{(x, \mu) \in A \times P} F_{\lambda_0}(x, \mu) > -\infty,
$$
one has
\begin{multline*}
  F_{\lambda}(x, \mu) = F_{\lambda_0}(x, \mu) + (\lambda - \lambda_0) \beta(d(\mu, \mu_0)) \ge \\
  \ge c + (\lambda - \lambda_0) \beta(\delta) \ge g(x^*, \mu_0) = F_{\lambda}(x^*, \mu_0).
\end{multline*}
Thus, for any $\lambda \ge \overline{\lambda} := \max\{ \lambda_1, \lambda_2 \}$ one has
$$
  F_{\lambda}(x, \mu) \ge F_{\lambda}(x^*, \mu_0) \quad \forall (x, \mu) \in A \times P
$$
or, in other words, the penalty function $F_{\lambda}$ attains a global minimum on $A \times P$ at the point 
$(x^*, \mu_0)$. Let $(\overline{x}, \overline{\mu}) \in A \times P$ be a different global minimizer of 
$F_{\lambda}$ on $A \times P$. Let us show that $\overline{\mu} = \mu_0$, provided 
$\lambda > \overline{\lambda}$, then one concludes that the penalty function $F_{\lambda}$ is globally exact.

Indeed, for any $\lambda > \overline{\lambda}$, $x \in A$ and $\mu \ne \mu_0$ one has
$$
  F_{\lambda}(x^*, \mu_0) = F_{\overline{\lambda}}(x^*, \mu_0) \le F_{\overline{\lambda}}(x, \mu) < F_{\lambda}(x, \mu),
$$
since $\beta(d(\mu, \mu_0)) > 0$. Hence $\overline{\mu} = \mu_0$ by virtue of the fact that 
$(\overline{x}, \overline{\mu})$ is a global minimizer of $F_{\lambda}$ on $A \times P$.  
\end{proof}

Let us also point out a connection between the calmness of the optimal value function $h$ and the calmness of the
perturbed problem ($\mathcal{P}_{\mu_0}$) at globally optimal solutions of the problem ($\mathcal{P}$) in the case when
the set $A$ is compact. 

\begin{theorem} \label{ThmCalmness}
Let the set $A$ be compact, and the function $g(x, \mu)$ be lower semicontinuous (l.s.c.) on $A \times B(\mu_0, r)$ for
some $r > 0$. Then for the optimal value function $h$ to be $\beta$-calm from below at $\mu_0$ it is necessary and
sufficient that the problem ($\mathcal{P}_{\mu_0}$) is $\beta$-calm at every globally optimal solution of the problem
($\mathcal{P}$).
\end{theorem}

\begin{proof}
Necessity. Fix a globally optimal solution $x^*$ of the problem ($\mathcal{P}$). From the definition of $\beta$-calmness
it follows that there exist $\lambda \ge 0$ and $r > 0$ such that
$$
  h(\mu) - h(\mu_0) \ge - \lambda \beta(d(\mu, \mu_0)) \quad \forall \mu \in B(\mu_0, r).
$$
From the fact that $x^*$ is a globally optimal solution of the problem ($\mathcal{P}$) it follows that 
$h(\mu_0) = g(x^*, \mu_0)$ due to the exact penalization condition on $g(x, \mu)$. Therefore
$$
  g(x, \mu) - g(x^*, \mu_0) \ge h(\mu) - h(\mu_0) \ge - \lambda \beta(d(\mu, \mu_0)) \quad 
  \forall (x, \mu) \in A \times B(\mu_0, r).
$$
Thus, the problem ($\mathcal{P}_{\mu_0}$) is $\beta$-calm at $x^*$.

Sufficiency. Taking into account the facts that $A$ is compact, and $g(x, \mu)$ is l.s.c. one gets that the function
$g(\cdot, \mu_0)$ attains a global minimum on the set $A$, and the set $A^*$ of all points of global minimum of
$g(\cdot, \mu_0)$ on $A$ is compact. Furthermore, from the exact penalization condition it follows that $A^*$ is also
the set of all globally optimal solutions of the problem ($\mathcal{P}$). Hence the problem ($\mathcal{P}_{\mu_0}$) is
$\beta$-calm at every $x^* \in A^*$. Therefore for any $x^* \in A^*$ there exist $\lambda(x^*) \ge 0$, $r(x^*) \ge 0$
and a neighbourhood $U(x^*)$ of $x^*$ such that
$$
  g(x, \mu) - g(x^*, \mu_0) \ge - \lambda(x^*) \beta(d(\mu, \mu_0)) \quad 
  \forall (x, \mu) \in U(x^*) \times B(\mu_0, r(x^*)).
$$
Applying the compactness of $A^*$ one obtains that there exist $x_1^*, x_2^*, \ldots, x_n^* \in A^*$ such that
$A^* \subset \bigcup_{k = 1}^n U(x^*_k)$. Denote
$$
  U = \bigcup_{k = 1}^n U(x_k^*), \quad \overline{\lambda} = \max_{k \in 1:n} \lambda(x_k^*), \quad
  \overline{r} = \min_{k \in 1:n} r(x_k^*).
$$
Then for any $x \in U$ one has
\begin{equation} \label{CalmNearOptSol}
  g(x, \mu) - h(\mu_0) \ge - \overline{\lambda} \beta(d(\mu, \mu_0)) \quad 
  \forall \mu \in B(\mu_0, \overline{r}),
\end{equation}
due to the fact that $g(x^*, \mu_0) = h(\mu_0) = \inf_{x \in A} g(x, \mu_0)$ for any $x^* \in A^*$.

Set $K = A \setminus U$. Since $A^* \subset U$, for any $x \in K$ one has $g(x, \mu_0) > h(\mu_0)$ (recall that
$A^*$ is the set of all global minimizers of $g(\cdot, \mu_0)$ on $A$). Consequently, applying the lower semicontinuity
of the function $g(x, \mu)$ one gets that for any $x \in K$ there exist $\delta(x) > 0$ and a neighbourhood $V(x)$ of
$x$ such that $g(y, \mu) > h(\mu_0)$ for all $(y, \mu) \in V(x) \times B(\mu_0, \delta(x))$. Observe that from the facts
that $U$ is open and $A$ is compact it follows that $K = A \setminus U$ is also compact. Applying this fact and the
inequality above it is easy to verify that there exists $\delta > 0$ such that $g(x, \mu) > h(\mu_0)$ for all 
$(x, \mu) \in K \times B(\mu_0, \delta)$. Therefore taking into account (\ref{CalmNearOptSol}) one gets that
$$
  g(x, \mu) - h(\mu_0) \ge - \overline{\lambda} \beta(d(\mu, \mu_0)) \quad 
  \forall (x, \mu) \in A \times B(\mu_0, \min\{ \delta, \overline{r} \}),
$$
which yields
$$
  h(\mu) - h(\mu_0) \ge - \overline{\lambda} \beta(d(\mu, \mu_0)) \quad 
  \forall \mu \in B(\mu_0, \min\{ \delta, \overline{r} \}).
$$
Thus, the optimal value function $h$ is $\beta$-calm from below at $\mu_0$.  
\end{proof}

Combining Theorems~\ref{ThmGlobExNSC} and \ref{ThmCalmness}, and Proposition~\ref{PrpLocalExNSC} one obtains that the
following result holds true.

\begin{corollary} \label{Crl_GlobExactCond}
Let $A$ be compact, $g(x, \mu)$ be l.s.c. on $A \times B(\mu_0, r)$ for some $r > 0$, and the function $\beta$ be
strictly increasing. Then the penalty function $F_{\lambda}$ is globally exact if and only if it is exact at every
globally optimal solution of the problem ($\mathcal{P}$), and there exists $\lambda_0 \ge 0$ such that
$F_{\lambda_0}(x, \mu)$ is bounded below on $A \times P$.
\end{corollary}

\section{Smooth Exact Penalty Functions}
\label{Sctn_ProofOfExactness}

Let us apply the theory developed in the previous section to the study of Huyer and Neumaier's exact penalty functions.
Let $X$ and $Y$ be metric spaces, $A \subset X$ be a nonempty set, and $\Phi \colon X \rightrightarrows Y$ be a given
set-valued mapping with closed images. For any subset $C \subset X$ and $x_0 \in X$ denote by 
$d(x_0, C) = \inf_{x \in C} d(x_0, x)$ the distance between $C$ and $x_0$. For any $y \in Y$ denote, as usual, 
$\Phi^{-1}(y) = \{ x \in X \mid y \in \Phi(x) \}$.

Fix an element $y_0 \in Y$, and consider the following optimization problem:
\begin{equation} \label{MathProg}
  \min f(x) \quad \text{subject to} \quad y_0 \in \Phi(x), \quad x \in A.
\end{equation}
Note that the set of feasible points of this problem has the form $\Omega = \Phi^{-1}(y_0) \cap A$.

Following the general technique proposed above and the method of smooth exact penalty functions \cite{WangMaZhou},
define $P = \mathbb{R}_+$, fix a non-decreasing function 
$\phi \colon \mathbb{R}_+ \cup \{ + \infty \} \to \mathbb{R}_+ \cup \{ +\infty \}$
such that $\phi(t) = 0$ iff $t = 0$, and introduce the perturbed objective function
$$
  g(x, \mu) = \begin{cases}
    f(x), & \text{ if } x \in \Omega, \mu = 0, \\
    + \infty, & \text{ if } x \notin \Omega, \mu = 0, \\
    f(x) + \frac{1}{\mu} \phi( d(y_0, \Phi(x))^2 ), & \text{ if } \mu > 0.
  \end{cases}
$$
Clearly, the function $g(x, \mu)$ satisfies the consistency condition and the exact penalization condition with 
$\mu_0 = 0$.

Introduce the penalty function
$$
  F_{\lambda}(x, \mu) = g(x, \mu) + \lambda \beta(\mu),
$$
where $\beta \colon \mathbb{R}_+ \to \mathbb{R}_+ \cup \{ + \infty \}$ is a non-decreasing function such that 
$\beta(\mu) = 0$ iff $\mu = 0$. Let us obtain sufficient conditions for $F_{\lambda}(x, \mu)$ to be exact. In order to
formulate these conditions, recall that a set valued mapping $\Phi$ is said to be \textit{metrically subregular} with
respect to the set $A$ with constant $a > 0$ at a point $(\overline{x}, \overline{y}) \in X \times Y$ with
$\overline{y} \in \Phi(\overline{x})$ and $\overline{x} \in A$, if there exists a neighbourhood $U$ of $\overline{x}$
such that
$$
  d(\Phi(x), \overline{y}) \ge a d(x, \Phi^{-1}(\overline{y}) \cap A) \quad \forall x \in U \cap A.
$$
Thus, $\Phi$ is metrically subregular with respect to the set $A$ iff the restriction of $\Phi$ to $A$ is metrically
subregular in the usual sense. See \cite{Aze, Kruger} and the references therein for the extensive study of metric
subregularity.

\begin{theorem} \label{Thm_LocExSmoothPenFunc}
Let $x^* \in \dom f$ be a point of local minimum of the problem (\ref{MathProg}), the function $f$ be Lipschitz
continuous near $x^*$, and the set-valued mapping $\Phi$ be metrically subregular with respect to the set $A$ with
constant $a > 0$ at $(x^*, y_0)$. Suppose also that the following assumptions are satisfied:
\begin{enumerate}
\item{there exist $\phi_0 > 0$ and $t_0 > 0$ such that $\phi(t) \ge \phi_0 t$ for any 
$t \in [0, t_0]$;
}

\item{there exist $\beta_0 > 0$ and $\overline{\mu} > 0$ such that $\beta(\mu) \ge \beta_0 \mu$ for any
$\mu \in [0, \overline{\mu}]$.
}
\end{enumerate}
Then the penalty function $F_{\lambda}(x, \mu)$ is exact at $x^*$. Moreover, one has
\begin{equation} \label{UpperEstimExPenParam}
  \lambda(x^*) \le \frac{L^2}{4 \phi_0 \beta_0 a^2},
\end{equation}
where $L \ge 0$ is a Lipschitz constant of $f$ near $x^*$.
\end{theorem}

\begin{proof}
Since $x^*$ is a point of local minimum of the problem (\ref{MathProg}), there exists $\rho > 0$ such that 
$f(x) \ge f(x^*)$ for any $x \in B(x^*, \rho) \cap \Omega$. Suppose, for a moment, that there exists 
$\delta \in (0, \rho)$ such that
\begin{equation} \label{LowerEstimLipFunc}
  f(x) - f(x^*) \ge - L d(x, \Omega) \quad \forall x \in B(x^*, \delta) \setminus \Omega,
\end{equation}
where $L \ge 0$ is a Lipschitz constant of $f$ near $x^*$. Then applying the metric subregularity of $\Phi$ with
respect to $A$, and the fact that the function $\phi$ is non-decreasing, one gets that for any $\mu > 0$ and 
$x \in B(x^*, r) \cap A$, where $r = \min\{ \delta, t_0 \}$, the following inequalities hold true:
\begin{multline*}
  g(x, \mu) - g(x^*, 0) = f(x) - f(x^*) + \frac{1}{\mu} \phi( d(y_0, \Phi(x))^2 ) \ge \\
  \ge - L d(x, \Omega) + \frac{\phi_0 a^2}{\mu} d(x, \Phi^{-1}(y_0) \cap A)^2 =
  - L d(x, \Omega) + \frac{\phi_0 a^2}{\mu} d(x, \Omega)^2.
\end{multline*}
Note that the function $h(t) = - L t + \phi_0 a^2 t^2 / \mu$ attains a global minimum at the point 
$\mu L / 2 \phi_0 a^2$, and 
$$
  h\left( \frac{\mu L}{2 \phi_0 a^2} \right) = - \frac{L^2}{4 \phi_0 a^2} \mu.
$$
Hence for any $\mu > 0$ and $x \in B(x^*, r) \cap A$ one has
\begin{equation} \label{CalmBelowPOF}
  g(x, \mu) - g(x^*, 0) \ge - \frac{L^2}{4 \phi_0 a^2} \mu.
\end{equation}
On the other hand, if $x \in B(x^*, r) \cap A$ and $\mu = 0$, then either $x \notin \Omega$ and 
$g(x, \mu) = + \infty \ge g(x^*, 0)$ or $x \in \Omega$ and $g(x, \mu) = f(x) \ge f(x^*) = g(x^*, 0)$ (recall that 
$r \le \delta < \rho$). Therefore the inequality (\ref{CalmBelowPOF}) is satisfied for any $x \in B(x^*, r) \cap A$ and 
$\mu \ge 0$, which yields that
$$
  F_{\lambda}(x, \mu) = g(x, \mu) + \lambda \beta(\mu) \ge g(x, \mu) + \lambda \beta_0 \mu \ge
  g(x^*, 0) =  F_{\lambda}(x^*, 0)
$$
for all $(x, \mu) \in B(x^*, r) \times [0, \overline{\mu}]$, and for any $\lambda \ge L^2 / 4 \phi_0 \beta_0 a^2$. Thus,
the penalty function $F_{\lambda}(x, \mu)$ is exact at $x^*$, and (\ref{UpperEstimExPenParam}) holds true.

It remains to show that inequality (\ref{LowerEstimLipFunc}) is valid for some $\delta > 0$. Indeed, 
fix $x \in B(x^*, \rho / 2) \setminus \Omega$. By the definition of the distance between a point and a set
there exists a sequence $\{ x_n \} \subset \Omega$ such that $d(x, x_n) \to d(x, \Omega)$ as $n \to \infty$.
Moreover, without loss of generality one can suppose that $d(x, x_n) \le \rho / 2$ for any $n \in \mathbb{N}$, since
$d(x, x^*) \le \rho / 2$ and $x^* \in \Omega$. Consequently, one has
$$
  d(x_n, x^*) \le d(x_n, x) + d(x, x^*) \le \frac{\rho}{2} + \frac{\rho}{2} = \rho,
$$
which implies that $f(x_n) \ge f(x^*)$. Therefore applying the Lipschitz continuity of $f$ near $x^*$ one obtains that
for any $n \in \mathbb{N}$ the following inequalities holds true
$$
  f(x) - f(x^*) = f(x) - f(x_n) + f(x_n) - f(x^*) \ge f(x) - f(x_n) \ge - L d(x, x_n).
$$
Passing to the limit as $n \to \infty$ one obtains the desired result.  
\end{proof}

Applying Corollary~\ref{Crl_GlobExactCond}, and the theorem above one can easily obtain sufficient conditions for the
penalty function $F_{\lambda}$ to be globally exact.

\begin{theorem}
Let the set $A$ be compact. Suppose that the following assumptions are satisfied:
\begin{enumerate}
\item{$f$ is l.s.c. on $A$, and locally Lipschitz continuous near globally optimal solutions of the problem
($\mathcal{P}$);
}

\item{$\Phi$ is metrically subregular with respect to the set $A$ at $(x^*, y_0)$ for any globally optimal solution
$x^*$ of the problem ($\mathcal{P}$);
}

\item{the mapping $x \to d(y_0, \Phi(x))$ is continuous on $A$;
}

\item{$\phi$ is l.s.c., and there exist $\phi_0 > 0$ and $t_0 > 0$ such that $\phi(t) \ge \phi_0 t$ for
any $t \in [0, t_0]$;
}

\item{$\beta$ is strictly increasing and there exist $\beta_0 > 0$ and $\overline{\mu} > 0$ such that 
$\beta(\mu) \ge \beta_0 \mu$ for any $\mu \in [0, \overline{\mu}]$.
}
\end{enumerate}
Then the penalty function $F_{\lambda}$ is globally exact.
\end{theorem}

\begin{proof}
Applying Theorem~\ref{Thm_LocExSmoothPenFunc}, and taking into account the assumptions of the theorem one obtains that
the penalty function $F_{\lambda}$ is exact at every globally optimal solution of the problem ($\mathcal{P}$). Since $f$
is l.s.c. on $A$, and the set $A$ is compact, then $f$ is bounded below on this set. Hence the function $g(x, \mu)$ is
bounded below on $A \times \mathbb{R}_+$, which implies that the penalty function $F_{\lambda}$ is also bounded below on
$A \times \mathbb{R}_+$ for any $\lambda \ge 0$.

Let us show that the function $g(x, \mu)$ is l.s.c. on $A \times \mathbb{R}_+$. Then with the use of
Corollary~\ref{Crl_GlobExactCond} one obtains the desired result.

For any $\varepsilon > 0$ introduce the function
$$
  g_{\varepsilon}(x, \mu) = f(x) + \frac{1}{\mu + \varepsilon} \phi(d(y_0, \Phi(x))^2) \quad 
  (x, \mu) \in A \times \mathbb{R}_+.
$$
Taking into account the fact that the function $x \to d(y_0, \Phi(x))$ is continuous on $A$, and $\phi$ is l.s.c., one
gets that the function $x \to \phi(d(y_0, \Phi(x))^2)$ is l.s.c. on $A$ as well. Hence and from the lower semicontinuity
of $f$ it follows that the function $g_{\varepsilon}(x, \mu)$ is l.s.c. on $A \times \mathbb{R}_+$. Note that
$$
  g(x, \mu) = \sup_{\varepsilon > 0} g_{\varepsilon}(x, \mu) \quad \forall (x, \mu) \in A \times \mathbb{R}_+.
$$
Therefore the function $g(x, \mu)$ is l.s.c. on $A \times \mathbb{R}_+$ as the supremum of a family of
l.s.c. functions.  
\end{proof}

Theorem~\ref{Thm_LocExSmoothPenFunc} can be modified to the case of more general functions $g(x, \mu)$ and
$F_{\lambda}(x, \mu)$. In particular, let
$$
  g(x, \mu) = \begin{cases}
    f(x), & \text{ if } x \in \Omega, \mu = 0, \\
    + \infty, & \text{ if } x \notin \Omega, \mu = 0, \\
    f(x) + \frac{1}{\mu^{\alpha}} \phi( d(y_0, \Phi(x))^2 ), & \text{ if } \mu > 0.
  \end{cases}
$$
and
$$
  F_{\lambda}(x, \mu) = g(x, \mu) + \lambda \beta(\mu),
$$
where $\alpha > 0$ (cf.~\cite{LinWuYu, LiYu, LinLoxton}). The following result holds true.

\begin{theorem}
Let $x^* \in \dom f$ be a point local minimum of the problem (\ref{MathProg}), the function $f$ be Lipschitz continuous
near $x^*$, and the set-valued mapping $\Phi$ be metrically subregular with respect to the set $A$ at $(x^*, y_0)$. 
Suppose that the following assumptions are satisfied:
\begin{enumerate}
\item{$\phi(t) \ge \phi_0 t^{\gamma}$ for any $t \in [0, t_0]$, and for some $\phi_0 > 0$, $\gamma > 0$ and $t_0 > 0$;
}

\item{$\beta(\mu) \ge \beta_0 \mu^{\sigma}$ for any $\mu \in [0, \overline{\mu}]$, and for some $\beta_0 > 0$, 
$\sigma > 0$ and $\overline{\mu} > 0$.
}
\end{enumerate}
Suppose also that
$$
  \gamma > \frac{1}{2}, \quad \sigma \le \frac{\alpha}{2 \gamma - 1}.
$$
Then the penalty function $F_{\lambda}(x, \mu)$ is exact at $x^*$. 
\end{theorem}

\begin{proof}
Arguing in the same way as in the proof of Theorem~\ref{Thm_LocExSmoothPenFunc} one can easily verify that there exists
$\Theta > 0$ such that 
$$
  g(x, \mu) - g(x^*, 0) \ge - \Theta \mu^{\frac{\alpha}{2 \gamma - 1}}	\quad
  \forall x \in B(x^*, r) \cap A \quad \forall \mu \ge 0,
$$
where $r > 0$ is sufficiently small. Then applying the assumption on the function $\beta$ one can check that 
$F_{\lambda}$ is exact at $x^*$, provided $\sigma \le \alpha / (2 \gamma - 1)$.  
\end{proof}

\bibliographystyle{abbrv}  
\bibliography{SmoothPenFunc_bibl}

\end{document}